\documentclass[12pt]{amsart}
\usepackage{amsmath,times,epsfig,amssymb,amsbsy,amscd,amsfonts,amstext,color,bm}
\usepackage{enumerate}
\usepackage{placeins}
\usepackage{array}
\usepackage{tabulary}
\usepackage{multirow}
\usepackage{graphicx}
\usepackage{hyperref}
\usepackage{hhline}
\usepackage{tabu}
\usepackage[table]{xcolor}
\usepackage{subcaption}
\usepackage[color,matrix,arrow]{xypic}
\hypersetup{
  colorlinks   = true, 
  urlcolor     = blue, 
  linkcolor    = blue, 
  citecolor   = red 
}

\theoremstyle{plain}
\newtheorem{theorem}{Theorem}[section]
\newtheorem{corollary}[theorem]{Corollary}
\newtheorem{lemma}[theorem]{Lemma}

\newtheorem{proposition}[theorem]{Proposition}

\makeatletter
    
    \@addtoreset{equation}{section}
  \makeatother

\theoremstyle{definition}
\newtheorem{definition}[theorem]{Definition}
\newtheorem{example}[theorem]{Example}

\theoremstyle{remark}
\newtheorem{remark}[theorem]{Remark}

\newcommand{\A}{\mathcal{A}}

\newcommand{\C}{\mathbb{C}}

\newcommand{\R}{\mathbb{R}}

\newcommand{\G}{\mathcal{G}}
\newcommand{\scR}{\mathcal{R}}
\newcommand{\scS}{\mathcal{S}}
\newcommand{\scT}{\mathcal{T}}
\newcommand{\scC}{\mathcal{C}}
\newcommand{\scD}{\mathcal{D}}
\newcommand{\scP}{\mathcal{P}}

\newcommand{\Z}{\mathbb{Z}}

\newcommand{\scH}{{\mathcal{H}}}

\newcommand{\M}{\mathcal{M}}

\newcommand{\lcm}{\operatorname{lcm}}
\newcommand{\Hom}{\operatorname{Hom}}
\newcommand{\quasi}{\operatorname{quasi}}

\renewcommand{\part}{\operatorname{par}}

\newcommand{\arith}{\operatorname{arith}}

\newcommand{\free}{\operatorname{free}}

\newcommand{\tot}{\operatorname{tot}}
\newcommand{\tor}{\operatorname{tor}}

\newcommand{\toric}{\operatorname{toric}}

\newcolumntype{K}[1]{>{\centering\arraybackslash}p{#1}}
\allowdisplaybreaks
\begin{document}

\title[Abelian Lie group arrangements]{Combinatorics of certain abelian Lie group arrangements and chromatic quasi-polynomials}

\date{\today}

\begin{abstract}
The purpose of this paper is twofold. 
Firstly, we generalize the notion of characteristic polynomials of hyperplane and toric arrangements to those of  certain abelian Lie group arrangements. 
Secondly, we give two interpretations for the chromatic quasi-polynomials and their constituents through subspace and toric viewpoints.

 \end{abstract}

\author{Tan Nhat Tran}
\address{Tan Nhat Tran, Department of Mathematics, Hokkaido University, Kita 10, Nishi 8, Kita-Ku, Sapporo 060-0810, Japan.}
\email{trannhattan@math.sci.hokudai.ac.jp}
\author{Masahiko Yoshinaga}
\address{Masahiko Yoshinaga, Department of Mathematics, Hokkaido University, Kita 10, Nishi 8, Kita-Ku, Sapporo 060-0810, Japan.}
\email{yoshinaga@math.sci.hokudai.ac.jp}



\date{\today}
\maketitle


\section{Introduction}
\label{sec:intro}
When a finite list $\A$ of integer vectors in $\Z^\ell$ is given, we may naturally associate to it a \emph{hyperplane arrangement} $\A(\R)$ in $\R^\ell$ and a \emph{toric arrangement} $\A(G)$ in the torus $G^\ell$ with $G$ is ${\Bbb S}^1$ or $\C^\times$. 
The study of a hyperplane/toric arrangement typically goes along with the study of its \emph{characteristic polynomial} as the polynomial carries combinatorial and topological information of the arrangement (e.g., \cite{OS80}, \cite{OT92}, \cite{Loo93}, \cite{DP05}). 
Many attempts have been made in order to compute and to make a broader understanding the characteristic polynomials (e.g., \cite{A96}, \cite{BS98},  \cite{KTT08},  \cite{ERS09},  \cite{Law11}, \cite{L12}, \cite{DM13}, \cite{BM14}, \cite{Y18W}, \cite{Y18L}). 
One of the first and classical ways is to define a polynomial in more than one variable which specializes to the characteristic polynomial. It is well-known that the characteristic polynomial of any hyperplane arrangement gets generalized to the \emph{Tutte polynomial} \cite{T54}, going back to Whitney (e.g., \cite[Theorem 2.4]{St07}). 
More recently, an arithmetical generalization of the ordinary Tutte polynomial, the \emph{arithmetic} Tutte polynomial was introduced \cite{L12} to extend the analysis on the toric case. 
For another interesting way of generalizing the polynomials, we would encode the information of the characteristic polynomials of these arrangements in a single quasi-polynomial, the \emph{characteristic quasi-polynomial} of $\A$ \cite{KTT08}. 
This quasi-polynomial was defined to evaluate the cardinality of the complement of the \emph{$q$-reduction arrangement} $\A(\Z/q\Z)$ in $(\Z/q\Z)^\ell$. 
Then the characteristic polynomials of the hyperplane and toric arrangements coincide with the first and the last constituents of the characteristic quasi-polynomial, respectively (\cite{KTT08}, \cite{LTY17}). 

In a more general setting, all of the concepts mentioned above can be redefined through a finite list $\A$ of elements in a finitely generated abelian group $\Gamma$. 
The notions of hyperplane, toric and $q$-reduction arrangements are unified by the notion of \emph{$G$-plexifications} $\A(G)$ \cite{LTY17}, which are defined by means of group homomorphisms from $\Gamma$ to certain abelian groups $G$. 
The $G$-plexifications when $G=F\times({\Bbb S}^1)^p\times \R^q,$ with $F$ is a finite abelian group are (non-trivial) examples of arrangements of \emph{abelian Lie groups}. 
In \cite{LTY17}, their topologies provided that $q>0$ have been described by the \emph{$G$-characteristic polynomials}. 
Studying problems related to this specific class of abelian Lie group arrangements is a source of our interest and suffices the purpose of generalizing the previous concepts.
The first result of the paper is obtained through the motivation of giving a combinatorial framework that describes the intersection patterns of $\A(G)$. 
Insprired by the pioneered work in \cite{L12}, we will associate to $\A(G)$ intersection posets and prove that its \emph{partial} and \emph{total} characteristic polynomials are also expressible in terms of the $G$-characteristic polynomials
 (Section \ref{sec:combinatorics}). 

The cardinality of the complement of $\A(\Z/q\Z)$ appeared as an analogue of the chromatic polynomial defined on a graph, and is called the \emph{chromatic quasi-polynomial} \cite{BM14}. 
We are also interested in a question that given a constituent of a chromatic polynomial how we can describe it in connection with arrangement characteristic polynomials?
Less is known, except for the first and the last. 
Some attempts were made to describe certain classes of the constituents appeared in \cite{Tan18}, \cite[\S10.3.3]{DFM17}.
The second result of the paper is two complete interpretations for the constituents through subspace and toric viewpoints. 
The subspace interpretation is obtained from the combinatorics of $\A(\R^{\dim(G)}\times \Z/k\Z)$, while the toric interpretation is obtained from the arithmetics of $\A({\Bbb S}^1)$ (or $\A(\C^\times)$) by appropriately extracting its intersection poset (Section \ref{sec:constituents}). 
 
The remainder of the paper is organized as follows. 
In Section \ref{sec:background}, we recall definitions and basic facts of the generalized toric arrangements, $G$-plexifications with the associated $G$-Tutte polynomials, and chromatic quasi-polynomials. 
In Section \ref{sec:combinatorics}, we define for  $\A(G)$ with $G=F\times({\Bbb S}^1)^p\times \R^q$ the total and partial intersection posets, and express the corresponding characteristic polynomials in terms of $G$-characteristic polynomials (Theorem \ref{thm:Lie}, Corollary \ref{cor:G-cha}). 
In Section \ref{sec:constituents}, we obtain the subspace interpretation (Corollary \ref{cor:Lie-chromatic}) of the chromatic quasi-polynomial immediately from the preceding section. 
Then we obtain reciprocity laws for the constituents with the aid of the reciprocity laws for integral arrangements (formula \eqref{eq:reciprocity}).
The toric interpretation (Theorem \ref{thm:2nd-interpret-quasi}) is derived from an important arithmetical result on the intersection poset (Lemma \ref{lem:k-components}). 
As a consequence, we provide another proof for the Chen-Wang's comparison of the coefficients of the chromatic quasi-polynomial.

\medskip

\section{Preliminaries} 
\label{sec:background}
Let us first fix some definitions and notations throughout the paper. 
Let $\Gamma$ be a finitely generated abelian group, and
let $\A\subseteq\Gamma$ be a finite list (multiset) of elements in $\Gamma$. 
For each sublist $\scS\subseteq \Gamma$, we denote by $r_\scS$ the \emph{rank} (as an abelian group)
of the subgroup $\langle\scS\rangle\le \Gamma$ generated by $\scS$. 
Given a group $K$, denote by $K_{\tor}$ the torsion subgroup of $K$. 
Denote $\scS^{\tor}:= \scS \cap \Gamma_{\tor}$.


\subsection{Generalized toric arrangements} 
\label{subsec:CQP}
Let $(\scP, \le_\scP)$ be a finite poset. 
The M\"{o}bius function $\mu_\scP$ of $\scP$ is the function $\mu_\scP:\scP \times \scP \longrightarrow\Z$ defined by
$$\mu_\scP(a,b):= 
 \begin{cases}
0 \quad\mbox{ if\,  $a\nleq_\scP b$}, \\
1 \quad\mbox{ if\,  $a=_\scP b$}, \\
-\sum_{a \le c <b}\mu_\scP(a,c) \quad\mbox{ if\,  $a<_\scP b$}.
\end{cases}
$$
A poset $\scP$ is said to be \emph{ranked} if for every $a \in \scP$, all maximal chains among those with $a$ as greatest element have the same length, denoted this common number by $\mathrm{rk}_\scP(a)$.

Now we briefly recall what has been known on combinatorics of generalized toric arrangements following \cite[\S5]{L12}.
Set $T:=\Hom(\Gamma, \mathbb{S}^1)$. 
In the following discussion, ${\Bbb S}^1$ and $\C^\times$ are interchangeable. 
Each $\alpha \in \A$ determines the subvariety of $T$ as follows $H_{\alpha}:=\{\varphi \in T\mid \varphi (\alpha)= 1\}$.
The collection $\scT(\A):=\{H_{\alpha}\mid\alpha\in\A\}$ is called the \emph{generalized toric arrangement} defined by $\A$ on $T$. 
In particular, when $\Gamma$ is free, $T$ is a torus and $\scT(\A)$ is called the \emph{toric arrangement}. 
To describe the combinatorics of $\scT(\A)$, we associate to it an \emph{intersection poset} $L_{\scT(\A)}$, which is the set of all the connected components of all the intersections of the subvarieties $H_{\alpha}$. 
The poset $L_{\scT(\A)}$ is ranked by the dimension of its elements (\emph{layers}).
The combinatorics is encoded in the \emph{characteristic polynomial} defined by
$$\chi_{\A}^{\toric}(t):= \sum_{\scC \in L_{\scT(\A)}} \mu(T^\scC, \scC)t^{\mathrm{dim}(\scC)},$$ 
where $T^\scC$ is the connected component of $T$ that contains $\scC$. To compute $\chi_{\A}^{\toric}(t)$ (in the same way as Whitney's theorem showing how the characteristic polynomial of a hyperplane arrangement is computed by the Tutte polynomial), Moci introduced the \emph{arithmetic Tutte polynomial}
$$T_{\A}^{\arith}(x, y):=
\sum_{\scS\subseteq \A}\#(\Gamma/\langle\scS\rangle)_{\tor}(x-1)^{r_\A-r_\scS}(y-1)^{\#\scS-r_\scS}.$$ 

  \begin{theorem}[\cite{L12}]
\label{thm:Moci's}
If $\Gamma$ is free and $0_\Gamma \notin \A$ (or even if $\Gamma$ is arbitrary with $\A^{\tor} = \emptyset$), then 
$$\chi_{\A}^{\toric}(t)=(-1)^{r_\A}\cdot t^{r_\Gamma-r_\A}\cdot T_{\A}^{\arith}(1-t,0).$$
\end{theorem}


\subsection{$G$-plexifications}
\label{subsec:G-Tutte} 
Let $G$ be an arbitrary abelian group.
We recall the notions of $G$-plexifications and $G$-Tutte polynomials of $\A$ following \cite[\S3]{LTY17}. 
We regard $T=\Hom(\Gamma, G)$ as our total group. 
For each $\alpha \in \A$, we define the \emph{$G$-hyperplane} associated to $\alpha$ as follows:
\begin{equation*}
H_{\alpha, G}:=\{\varphi \in T \mid \varphi (\alpha)= 0\} \le T. 
\end{equation*}
Then the \emph{$G$-plexification} $\A(G)$ of $\A$ is the collection of the subgroups $H_{\alpha, G}$
$$\A(G):=\{H_{\alpha, G}\mid\alpha\in\A\}.$$
The \emph{$G$-complement} $\M(\A; \Gamma, G)$ of $\A(G)$ is defined by
\begin{equation*}
\M(\A; \Gamma, G):=T\smallsetminus\bigcup_{\alpha\in\A}H_{\alpha, G}. 
\end{equation*}
 
In what follows, we assume further that $G$ is torsion-wise finite i.e., $G[d]:=\{x\in G\mid d\cdot x=0\}$ is finite for all $d\in\Z_{>0}$. 
The \emph{$G$-multiplicity} $m(\scS; G)$ for each $\scS\subseteq \A$ is defined by 
\begin{equation*}
m(\scS; G):=
\#\Hom\left((\Gamma/\langle\scS\rangle)_{\tor}, G\right). 
\end{equation*}

\begin{definition}\quad
\label{def:main}
\begin{enumerate}[(1)]
\item  
The \emph{$G$-Tutte polynomial} $T_{\A}^{G}(x, y)$ of $\A$ is defined by 
\begin{equation*}
T_{\A}^{G}(x, y):=
\sum_{\scS\subseteq \A}m(\scS; G)(x-1)^{r_\A-r_\scS}(y-1)^{\#\scS-r_\scS}. 
\end{equation*}
\item 
The \emph{$G$-characteristic polynomial} $\chi_{\A}^G(t)$ of $\A$ is defined by 
\begin{equation*}
\chi_{\A}^G(t):=(-1)^{r_\A}\cdot t^{r_{\Gamma}-r_\A}\cdot T_{\A}^{G}(1-t, 0). 
\end{equation*}
\end{enumerate}
\end{definition}

\begin{proposition}
\label{prop:gen-quasi-gcd1}
The leading coefficient of $\chi^G_{\A}(t)$ equals $\#\M(\A^{\tor}; \Gamma_{\tor}, G)$. 
\end{proposition}

Various specializations of the $G$-plexifications and $G$-Tutte polynomials have appeared in the literature which we refer the reader to \cite{LTY17} for more details.
In particular, the real hyperplane arrangements and generalized toric arrangements are $G$-plexifications by viewing $G=\R$ and $G=\mathbb{S}^1$ (or $G=\C^\times$), respectively. 


\subsection{Chromatic quasi-polynomials} 
\label{subsec:ChQP} 
For each $\scS\subseteq \A$, by the Structure Theorem, we may write $\Gamma/\langle\scS\rangle\simeq\bigoplus_{i=1}^{n_{\scS}}\Z/d_{\scS, i}\Z\oplus\Z^{r_{\Gamma}-r_{\scS}}$ where $n_{\scS}\geq 0$ and $1<d_{\scS, i}|d_{\scS, i+1}$. 
The \emph{LCM-period} $\rho_{\A}$ of $\A$ is defined by 
\begin{equation*}
\label{eq:LCM-period}
\rho_\A:=\lcm(d_{\scS, n_{\scS}}\mid\scS\subseteq \A). 
\end{equation*} 
It is proved in \cite{BM14} that $\#\M(\A; \Gamma, \Z/q\Z)$ is a quasi-polynomial in $q\in\Z_{>0}$ for which $\rho_{\A}$ is a period. 
The quasi-polynomial is called the \emph{chromatic quasi-polynomial} of $\A$, and denoted by $\chi^{\quasi}_{\A}(q)$.
More precisely, there exist polynomials $f_{\A}^k(t)\in\Z[t]$ ($1 \le k \le \rho_\A$),  called the \emph{$k$-constituents}, such that 
for any positive integer $q$, 
\begin{equation*}
\#\M(\A; \Gamma, \Z/q\Z) =f_{\A}^k(q), 
\end{equation*}
where $q\equiv k\bmod \rho_\A$. 

The chromatic quasi-polynomial is precisely the $\Z/q\Z$-characteristic polynomial in variable $q$, i.e., $\chi^{\quasi}_{\A}(q)=\chi^{\Z/q\Z}_{\A}(q)$ (e.g., \cite[Theorem 5.4]{LTY17}), and also the Chen-Wang's quasi-polynomial \cite{Tan18}. 
In particular, when $\Gamma=\Z^\ell$, the $\Z/q\Z$-plexification is the $q$-reduction arrangement defined on $\A$ with $\chi^{\quasi}_{\A}(q)$ is the characteristic quasi-polynomial in the sense of \cite{KTT08}.
A partial description of the constituents in connection with arrangement theory is known.
  \begin{theorem}[\cite{Tan18}]
\label{thm:1-const}
 $$f_{\A}^1(t)= 
\begin{cases}
0 \quad\mbox{ if $\A^{\tor} \ne \emptyset$}, \\
\chi_{\A(\R)}(t)  \quad\mbox{ if $\A^{\tor} = \emptyset$}.
\end{cases}
$$
\end{theorem}

  \begin{theorem}[\cite{LTY17}]
\label{thm:last-const}
Assume that $\Gamma$ is free and $0_\Gamma \notin \A$. Then
$$f^{\rho_\A}_{\A}(t)=\chi_{\A}^{\toric}(t).$$
\end{theorem}

\section{The combinatorics}
\label{sec:combinatorics}
Unless otherwise stated, throughout this section, we assume that $G=(\mathbb{S}^1)^p\times \R^q\times F$ with $g:=\dim(G)=p+q\ge0$ and $F$ is a finite abelian group.  
For each $\scS \subseteq \A$, by \cite[Proposition 3.6]{LTY17}, we have
 \begin{equation}
 \label{eq:intersections}
\begin{aligned}
H_{\scS, G}
&:=\bigcap_{\alpha\in\scS}H_{\alpha, G} \\
&\simeq \Hom((\Gamma/\langle\scS\rangle)_{\tor}, G)\times F^{r_\Gamma-r_{\scS}}\times\left((\mathbb{S}^1)^p\times\R^q\right)^{r_{\Gamma}-r_{\scS}}.
\end{aligned} 
 \end{equation}
We agree that $T:=H_{\emptyset, G}$.  
Each connected component of $H_{\scS, G}$ is isomorphic to $\left((\mathbb{S}^1)^p\times\R^q\right)^{r_{\Gamma}-r_{\scS}}$. 
If either $r_{\Gamma}=0$ or $g=0$, it can be identified with a point. 
The set of the connected components of $H_{\scS, G}$ is denoted by  $\mathrm{cc}(H_{\scS, G})$.
The following lemma is somewhat more general than \cite[Lemma 5.4]{L12}.

\begin{lemma}
\label{lem:components}
$\#\mathrm{cc}(H_{\scS, G}) = m(\scS; G) \cdot (\#F)^{r_\Gamma-r_{\scS}}$.
\end{lemma}

Most of the main concepts in this section are defined by inspiration of \cite[\S5]{L12} and \cite[\S7]{LTY17}.

\begin{definition}\quad
\label{def:total}
\begin{enumerate}[(1)]
\item  
The \emph{total} intersection poset of $\A(G)$ is defined by
\begin{equation*}
L=L^{\tot}_{\A(G)}:=\{\mbox{connected components of nonempty $H_{\scS, G}$} \mid \scS \subseteq \A\},
\end{equation*}
whose elements, called \emph{layers}, are ordered by reverse inclusion ($\scD \le_L \scC$ if $\scD \supseteq \scC$). 
\item 
The \emph{total} characteristic polynomial of $\A(G)$ is defined by 
\begin{equation*}
\chi_{\A(G)}^{\tot}(t):= \sum_{\scC \in L}\mu(T^\scC, \scC)t^{\mathrm{dim}(\scC)}.
\end{equation*}
Here $T^\scC$ is the connected component of $T$ that contains $\scC$, and $\mu:=\mu_L$. 
\end{enumerate}
\end{definition}

The set of minimal elements of $L$ is exactly $\mathrm{cc}(T)$. 
The connected components of $H_{\A, G}$ are maximal elements of $L$ but the converse is not necessarily true.
For each $\scC\in L$, set
$$\scR(\scC):=\{ \scS \subseteq \A  \mid \scC \in \mathrm{cc}(H_{\scS, G})\}.$$
One observes that $\mathrm{dim}(\scC)=\mathrm{dim}(H_{\scS, G})=g(r_\Gamma-r_\scS)$ for every $\scS \in \scR(\scC)$. 
The \emph{localization} of $\A$ with respect to $\scC$ is defined by 
$$\A_{\scC}:=\{\alpha \in \A \mid \scC  \subseteq H_{\alpha, G}\}.$$ 
Stated differently, $\A_{\scC}$ is the unique maximal element  of $\scR(\scC)$ in the sense that $\scS \subseteq \A_{\scC}$ for every $\scS \in \scR(\scC)$. 
We also can write
 \begin{equation}
 \label{eq:differently}
 \scR(\scC)=\{ \scS \subseteq \A_\scC  \mid r_{\scS} = r_{\A_\scC}\}.
 \end{equation}
Thus $L$ is a ranked poset with a rank function is given by $\mathrm{rk}_L(\scC):= r_{\A_\scC}=\mathrm{codim}(\scC)/g$ for $\scC\in L$. 
 
We are interested in a particular subset of $\mathrm{cc}(T)$,
\begin{align*}
\mathrm{scc}(T)
&:=\{T_i \in \mathrm{cc}(T)\mid (\A_{T_i})^{\tor} = \emptyset\} \\
&=\mathrm{cc}(T)\smallsetminus\bigcup_{\alpha\in\A^{\tor}}\mathrm{cc}(H_{\alpha, G}).
\end{align*}
By using the Inclusion-Exclusion principle,
 \begin{equation}
 \label{eq:card-scc}
\#\mathrm{scc}(T)=\#\M(\A^{\tor}; \Gamma_{\tor},G)  \cdot (\#F)^{r_\Gamma}.
 \end{equation}

 \begin{definition}\quad
\label{def:partial}
\begin{enumerate}[(1)]
\item  
The \emph{partial} intersection poset of $\A(G)$ is defined by
\begin{equation*}
L^{\part}:=\{\scC \in L\mid T^\scC \in   \mathrm{scc}(T)\},
\end{equation*}
with the M\"{o}bius function of $L^{\part}$ is the restriction of $\mu$ i.e., $\mu_{L^{\part}}=\left.\mu\right|_{L^{\part} \times L^{\part}}$.
\item 
The \emph{partial} characteristic polynomial of $\A(G)$ is defined by 
\begin{equation*}
\chi_{\A(G)}^{\part}(t):= \sum_{\scC \in L^{\part}} \mu(T^\scC, \scC)t^{\mathrm{dim}(\scC)}.
\end{equation*}

\end{enumerate}
\end{definition}

In other words, $L^{\part}$ is the dual order ideal (e.g., \cite[\S3.1]{St86}) of $L$ generated by $ \mathrm{scc}(T)$. 
It follows from the definition above that $\chi_{\A(G)}^{\part}(t)=0$ if $\mathrm{scc}(T) =\emptyset$.
\begin{remark}
\label{rem:Lie-par}
Removing from $\A(G)$ the hyperplanes $H_{\alpha, G}$ with $\alpha \in \A^{\tor}$ does not affect the structure of the poset  i.e., 
$$L_{\A(G)}=L_{(\A\smallsetminus \A^{\tor})(G)}= L_{(\A\smallsetminus \A^{\tor})(G)}^{\part}.$$
As a consequence, 
$$\chi^{\tot}_{\A(G)}(t)=\chi^{\tot}_{(\A\smallsetminus \A^{\tor})(G)}(t)= \chi_{(\A\smallsetminus \A^{\tor})(G)}^{\part}(t).$$
In particular, $\chi_{\A(G)}^{\tot}(t)=\chi_{\A(G)}^{\part}(t)$ if $\A^{\tor} = \emptyset$.
\end{remark}

In the lemma below, we generalize the result in \cite[Lemma 5.5]{L12} as we include the possibility $\A^{\tor} \ne \emptyset$. 
\begin{lemma}
\label{lem:key-Lie}
If $\scC\in L$, then  
 \begin{equation*}
 \label{eq:key-Lie}
 \sum_{\scS \in \scR(\scC)}(-1)^{\#\scS}=
 \begin{cases}
 \mu(T^\scC, \scC)\quad\mbox{ if\,  $\scC \in L^{\part}$}, \\
0 \quad\mbox{ if\,  $\scC \notin L^{\part}$}.
\end{cases}
 \end{equation*}
\end{lemma}
 \begin{proof} 
The proof of the first line is processed by induction on $\mathrm{rk}_L(\scC)$, which runs essentially the same as that of \cite[Lemma 5.5]{L12}. 
Note that $\A_{\scD} \subseteq \A_\scC$ whenever $\scC\subseteq\scD\subseteq T^{\scC}$. 
If $(\A_{T^\scC})^{\tor} \ne \emptyset$ then $(\A_{\scC})^{\tor}  \ne \emptyset$.
The remaining part of the formula follows from Theorem \ref{thm:1-const}. 
Indeed, by \eqref{eq:differently}
$$ \sum_{\scS \in \scR(\scC)}(-1)^{\#\scS}= \sum_{\substack{\scS\subseteq \A_\scC\\  r_{\scS} = r_{\A_\scC}}}(-1)^{\#\scS}$$
equals the coefficient of $t^{r_\Gamma-r_{\A_\scC}}$ in $f^1_{\A_\scC}(t)$, which is $0$.
   \end{proof}
   
    \begin{corollary} 
 \label{cor:alternate-sign} 
The M\"{o}bius function of $L$ strictly alternates in sign. That is, for all $\scC \in L$,
 $$(-1)^{\mathrm{rk}_L(\scC)} \mu(T^\scC, \scC)>0.$$
  \end{corollary}
  \begin{proof}
Consider $\scC \in L^{\part}$. 
Note that $f^1_{\A_\scC}(t)=\chi_{(\A_\scC)(\R)}(t)=\sum_{j=r_\Gamma-r_{\A_\scC}}^{r_\Gamma}b_jt^j$ with $(-1)^{r_\Gamma-j}b_j>0$ for all $j$ (e.g., \cite[Corollary 3.5]{St07}).
By Proof of Lemma \ref{eq:key-Lie}, $\mu(T^\scC, \scC)$ is equal to the coefficient of $t^{r_\Gamma-r_{\A_\scC}}$ in $f^1_{\A_\scC}(t)$, which strictly alternates in sign i.e., $(-1)^{r_{\A_\scC}} \mu(T^\scC, \scC)>0$.
If $\scC \notin L^{\part}$, we consider $\A\smallsetminus \A^{\tor}$ instead of $\A$ as argued in Remark \ref{rem:Lie-par}.
\end{proof}

The main idea of the proof below is very similar to the one used in \cite[Theorem 5.6]{L12}. 
We include it with a detailed proof for the sake of completeness.  
   \begin{theorem}
\label{thm:Lie}
 Let $G=(\mathbb{S}^1)^p\times \R^q\times F$ with $g=p+q\in \Z_{>0}$.  
Then
 $$\chi_{\A(G)}^{\part}(t)= \chi_{\A}^G\left( \#F\cdot t^g\right).$$
\end{theorem}
 \begin{proof}
We must prove that
 $$  \sum_{\scC \in L^{\part}} \mu(T^\scC, \scC)t^{\mathrm{dim}(\scC)}=\sum_{\scS\subseteq \A}(-1)^{\#\scS}
m(\scS; G) \cdot (\#F)^{r_\Gamma-r_{\scS}}\cdot t^{g(r_{\Gamma}-r_\scS)}. $$
It is equivalent to proving that for all $k=r_{\Gamma}-r_\A, \ldots, r_{\Gamma}$,
 $$  \sum_{\substack{\scC \in L^{\part}  \\  gk=\mathrm{dim}(\scC)}}  \mu(T^\scC, \scC) =\sum_{\substack{\scS\subseteq \A\\  k=r_{\Gamma}-r_\scS}} (-1)^{\#\scS}
m(\scS; G)\cdot (\#F)^{r_\Gamma-r_{\scS}}. $$
We have
\begin{align*}
 \sum_{\substack{\scC \in L^{\part} \\  gk=\mathrm{dim}(\scC)}}  \mu(T^\scC, \scC) 
& =  \sum_{\substack{\scC \in L\\  gk=\mathrm{dim}(\scC)}} \sum_{\scS \in \scR(\scC)}(-1)^{\#\scS}  \\
&= \sum_{\substack{\scS\subseteq \A\\  r_\scS=r_{\Gamma}-k}}\left( \sum_{\scC\in  \mathrm{cc}(H_{\scS, G})}1\right) (-1)^{\#\scS} \\
&=\sum_{\substack{\scS\subseteq \A\\  k=r_{\Gamma}-r_\scS}} (-1)^{\#\scS}
m(\scS; G)\cdot (\#F)^{r_\Gamma-r_{\scS}}.
\end{align*}
We have applied Lemma \ref{lem:key-Lie} in the first equality, switched roles of sums in the second equality, and used Lemma \ref{lem:components} in the last equality.
\end{proof}

 \begin{corollary} 
 \label{cor:G-cha} 
 Let $G=(\mathbb{S}^1)^p\times \R^q\times F$ with $g=p+q\in \Z_{>0}$. Then 
$$
\chi_{\A(G)}^{\tot}(t)
=
\chi_{\A\smallsetminus \A^{\tor}}^G\left( \#F\cdot t^g\right).
$$
  \end{corollary}
  \begin{proof}
It follows from Theorem \ref{thm:Lie}  and Remark \ref{rem:Lie-par}.
\end{proof}
\begin{remark}
\label{rem:g=0}
Although either Theorem \ref{thm:Lie} or Corollary \ref{cor:G-cha} may not be valid when $g=0$, there is no loss of information in these formulations. 
Namely, $\chi_{\A(G)}^{\part}(t)=\#\mathrm{scc}(T)$, and by equality \eqref{eq:card-scc} and Proposition \ref{prop:gen-quasi-gcd1}, this equals the ``leading part" of $\chi_{\A}^G\left( \#F\right)$ (the value of the leading term of $\chi_{\A}^G\left( t\right)$ evaluated at $\#F$). 
Similarly, $\chi_{\A(G)}^{\tot}(t)=\#\mathrm{cc}(T)$, which is equal to the leading part of $\chi_{\A\smallsetminus \A^{\tor}}^G\left( \#F\right)$.
\end{remark}

\begin{remark}
\label{rem:recover-Moci's}
Note that when $G=\mathbb{S}^1$ (or $G=\C^\times$ if the dimension is defined over $\C$) and $\A^{\tor}=\emptyset$, $\chi_{\A(G)}^{\part}(t)=\chi_{\A(G)}^{\tot}(t)=\chi_{\A}^{\toric}(t)$. 
The result of Moci (Theorem \ref{thm:Moci's}) is a special case of Corollary \ref{cor:G-cha}.
\end{remark}


  \section{The constituents}
\label{sec:constituents}

\subsection{Via subspace viewpoint} 
\label{subsec:via-subspace}

Our first result in this section is the interpretation for chromatic polynomials and their constituents through the real subspace arrangement viewpoint. 
Combining Theorem \ref{thm:Lie} with the property of the chromatic polynomials (e.g., \cite[Proposition 3.6]{Tan18}), we obtain

 \begin{corollary} 
 \label{cor:Lie-chromatic} 
 Let $G = \R^g\times \Z/q\Z$ with $g >0$, $q>0$. Then 
$$\chi_{\A(G)}^{\part}\left( q\right)=\chi^{\quasi}_{\A}(q^{g+1}).$$ 
  \end{corollary}

 \begin{corollary} 
 \label{cor:Lie-constituents} 
 Let $G = \R^g\times \Z/k\Z$ with $g > 0$ and $1 \le k \le\rho_{\A}$. 
 Then 
$$\chi_{\A(G)}^{\part}\left( t \right)=f^k_{\A}(k\cdot t^g).$$ 
  \end{corollary}
  
Let us explain Corollary \ref{cor:Lie-constituents} in more detail. 
For nontriviality, we assume that $\mathrm{scc}(T) \ne \emptyset$ (e.g., $\#\A^{\tor}\le 1<\#\A$), and $r_\Gamma>0$. 
Each connected component of $T=\Hom(\Gamma, \R^g\times \Z/k\Z)$ is isomorphic to $\R^{gr_\Gamma}$. 
For each $T_i \in \mathrm{scc}(T)$, the poset $L_i=\{\scC \in L\mid \scC \subseteq T_i \}$ is isomorphic to the total (or equivalently, partial) intersection poset of a $\R^g$-plexification $\G_i$ in $\R^{gr_\Gamma}$ (or \emph{$g$-plexification} in the sense of \cite[\S5.2]{Bj94})), with each $\G_i$ is possibly empty and defined over the integers. 
Thus after a rescaling of variable, each constituent records the summation of the total characteristic polynomials of the $\G_i$'s i.e.,
 \begin{equation}
 \label{eq:rescale-sum}
f^k_{\A}(kt^q) =\sum_{T_i \in \mathrm{scc}(T)}  \chi^{\tot}_{\G_i}(t).
 \end{equation}

\begin{remark}
\label{rem:familiar}
In particular, when $g=1$, each $\G_i$ becomes an integral hyperplane arrangement $\scH_i$ \cite[Proposition 4.5]{Tan18}. 
The conclusion related to the first constituent ($k=1$) in Corollary \ref{cor:Lie-constituents} is the same as that stated in Theorem \ref{thm:1-const}. 
In particular, if $\Gamma = \Z^{\ell}$, each hyperplane $H_{\alpha,\R\times \Z/k\Z}$ in $T$ can be identified with $H_{\alpha,\R} \times H_{\alpha,\Z/k\Z}$ in $\R^\ell \times (\Z/k\Z)^\ell$. 
Each arrangement $\scH_i$ turns out to be a subarrangement of $\A(\R)$, and in which components of $T$ that the components of $H_{\alpha,\R\times \Z/k\Z}$ locate depends on the arithmetics of the list $\A$.
Therefore the Hasse diagram of each $L_i$ is a subgraph of that of $L_{\A(\R)}$. 
\end{remark}
\begin{example}
\label{ex:caculation-special}
Let $\Gamma=\Z^2$, $\A=\{\alpha, \beta, \gamma\}\subsetneq\Z^2$ with $\alpha=(-1, {1})$, $\beta=(0, {2})$, and $\gamma=(0, 4)$.
Then 
$$
\chi^{\quasi}_{\A}(q) = 
\begin{cases}
q^2-2q+1  \quad\mbox{ if $\gcd(q,4)=1$}, \\
q^2-3q+2 \quad\mbox{ if $\gcd(q,4)=2$}, \\
q^2-5q+4 \quad\mbox{ if $\gcd(q,4)=4$}.
\end{cases}
$$
Set $G_k:=\R\times \Z/k\Z$ with $k\in \{1,2,4\}$. 
The Hasse diagrams of $L_{\A(G_k)}$ are drawn in Figures \ref{Fig1a}, \ref{Fig1b}, \ref{Fig1c}. 
The total characteristic polynomials $\chi_{\A(G_k)}^{\tot}(t)$ are computed according to the ``$\times n$", indicator of the number of isomorphic Hasse diagrams of $L_i$'s.
 \begin{figure}[h] 
\begin{subfigure}{.5\textwidth}
\[
 \xymatrixrowsep{.5cm}
\xymatrixcolsep{.1cm}
 \xymatrix{
& \bullet &\\
\bullet \ar@{-}[ur]  \ar@{-}[dr] && \bullet \ar@{-}[ul]  \ar@{-}[dl] \\
&\bullet&
}
\]  
\caption*{$\times 1$}  
\end{subfigure}
\caption{$\chi_{\A(G_1)}^{\tot}(t)=t^2-2t+1=f^1_{\A}(t).$}
\label{Fig1a}
\end{figure}

 \begin{figure}[h]     
\begin{subfigure}{.3\textwidth}
\[
 \xymatrixrowsep{.5cm}
\xymatrixcolsep{.1cm}
 \xymatrix{
& \bullet &\\
\bullet \ar@{-}[ur]  \ar@{-}[dr] && \bullet \ar@{-}[ul]  \ar@{-}[dl] \\
&\bullet&
}
\]  
\caption*{$\times2$}  
\end{subfigure}%
\begin{subfigure}{.3\textwidth}
\[
 \xymatrixrowsep{.5cm}
 \xymatrix{
\bullet \ar@{-}[d] \\
\bullet
}
\]  
  \caption*{$\times2$}
\end{subfigure}
\caption{$\chi_{\A(G_2)}^{\tot}(t)=4t^2-6t+2=f^2_{\A}(2t).$}
\label{Fig1b}
\end{figure}
 
  \begin{figure}[h] 
\begin{subfigure}{.3\textwidth}
\[
 \xymatrixrowsep{.5cm}
\xymatrixcolsep{.1cm}
 \xymatrix{
& \bullet &\\
\bullet \ar@{-}[ur]  \ar@{-}[dr] && \bullet \ar@{-}[ul]  \ar@{-}[dl] \\
&\bullet&
}
\]  
\caption*{$\times4$}  
\end{subfigure}%
\begin{subfigure}{.3\textwidth}
\[
 \xymatrixrowsep{.5cm}
 \xymatrix{
\bullet \ar@{-}[d] \\
\bullet
}
\]  
  \caption*{$\times12$}
\end{subfigure}
\caption{$\chi_{\A(G_4)}^{\tot}(t)=16t^2-20t+4=f^4_{\A}(4t).$}
\label{Fig1c}
\end{figure}
 \end{example}
 
Now we give a discussion on \emph{reciprocity laws} for $\chi^{\quasi}_{\A}(q)$. 
After the works of \cite{CW12} and \cite{Tan18}, we know that $(-1)^{r_\Gamma}\chi^{\quasi}_{\A}(-q) \ge 0$ for all $q \in \Z_{< 0}$. 
Also, this fact can be derived from formula \eqref{eq:quasi-minus} in this paper.
It is natural to ask whether the evaluations $(-1)^{r_\Gamma}\chi^{\quasi}_{\A}(-q)$ have any combinatorial/enumerative meaning.
A partial answer is probably well-known when $\Gamma = \Z^{\ell}$ that $(-1)^{\ell}\chi^{\quasi}_{\A}(-q)$ can be expressed in terms of the Ehrhart quasi-polynomial of an ``inside-out" polytope \cite{BZ06}.
The construction of the polytope and the hyperplanes cutting through it can be found in \cite[\S2.2]{KTT08} (with some modification). 

Owning to equality \eqref{eq:rescale-sum} and Remark \ref{rem:familiar}, we can give an answer to the aforementioned question. 
For nontriviallity, we assume that  $r_\Gamma>0$.
The reciprocity laws for the characteristic polynomial of an integral arrangement have been formulated by several methods \cite{A10}, \cite[\S7]{BS18}, \cite[\S4]{W15}. 
Thus the reciprocity laws for any (nonzero) constituent $f^k_{\A}(t)$ can be obtained from the reciprocity laws of the polynomials $\chi_{\scH_i}(t)$ as follows:
 \begin{equation}
 \label{eq:reciprocity}
(-1)^{r_\Gamma}f^k_{\A}(-t) =\sum_{i} (-1)^{r_\Gamma}\chi_{\scH_i}\left( \frac{-t}k\right).
 \end{equation}

\subsection{Via toric viewpoint} 
\label{subsec:via-toric}
We may expect that if there exists a ``nicer" expression to describe every constituent without making any rescaling of variable. 
It turns out that such expression can be obtained from the toric arrangement by appropriately extracting its poset of layers.
Now let us turn to the second interpretation via toric arrangement viewpoint. 
In the remainder of this section, we assume that $G$ is either $\mathbb{S}^1$ or $\C^\times$. 
We retain the notation of the total group $T =\Hom(\Gamma,  G)$, with the identity is denoted by $\textbf{1}$. 
For each $k \in \Z$, consider the homomorphism
$$E_k:T \longrightarrow T\quad \mbox{via} \quad \varphi \mapsto \varphi^k :=\varphi \cdots\varphi .$$

\begin{definition}\quad
\label{def:k-total}
\begin{enumerate}[(1)]
\item  
For each $k \in \Z$,  the \emph{$k$-total} intersection poset of $\A(G)$ is defined by
  \begin{equation*}
 \label{eq:subposet-equiv}
 L[k]= \{\scC \in L \mid \textbf{1} \in E_k(\scC)\}.
 \end{equation*}
\item 
The \emph{$k$-total} characteristic polynomial of $\A( G)$ is defined by 
$$\chi_{\A( G)}^{\textrm{\textit{k}-tot}}(t):= \sum_{\scC \in L[k]} \mu(T^\scC, \scC)t^{\mathrm{dim}(\scC)}.$$ 
\end{enumerate}
\end{definition}

The cover relation in $L$ is preserved in $L[k]$ i.e., if $\scC$ covers $\scD$ in $L$ and $\scC\in L[k]$ then $\scD\in L[k]$, which implies that $L[k]$ is an \emph{order ideal} (e.g., \cite[\S3.1]{St86}). 
For each $\scS \subseteq\A$, note that $H_{\scS,  G}$ is a subtorus of $T$ whose each connected component is isomorphic to the torus $G^{r_\Gamma-r_\scS}$. 
Let $\scC_\scS^\textbf{1} \in \mathrm{cc}(H_{\scS,  G})$ be the \emph{identity component} of $H_{\scS,  G}$, that is, the connected component that contains $\textbf{1}$. 
Thus $\mathrm{cc}(H_{\scS,  G})$ can be identified with the quotient group $H_{\scS,  G}/\scC_\scS^\textbf{1}$. 
In the lemma below, we generalize \cite[Lemma 5.4]{L12} in an arithmetical manner.
\begin{lemma}
\label{lem:k-components}
Fix $k\in\Z_{>0}$. For each $\scS \subseteq\A$, we have
$$\# \left( \mathrm{cc}(H_{\scS,  G})\cap L[k] \right) = m(\scS; \Z/k\Z).$$
\end{lemma}
 \begin{proof}
For each $\scS \subseteq\A$, the homomorphism $E_k$ induces the endomorphism $\overline{E_k}$ of $H_{\scS,  G}/\scC_\scS^\textbf{1}$ with
 \begin{equation}
 \label{eq:bar}
 \ker(\overline{E_k})=\mathrm{cc}(H_{\scS,  G})\cap L[k].
 \end{equation}
Using the identification $H_{\scS,  G} = \Hom(\Gamma/\langle\scS\rangle, G)$ and a decomposition $\Gamma/\langle\scS\rangle = (\Gamma/\langle\scS\rangle)_{\tor}\oplus (\Gamma/\langle\scS\rangle)_{\free}$, we can write 
$$ 
\scC_\scS^\textbf{1} =  \{\varphi \in  H_{\scS,  G} \mid \varphi(x)=1, \,\forall x \in \left(\Gamma/\langle \scS\rangle\right)_{\tor} \}.
$$ 
Applying the exact functor $\Hom(\textendash, G)$ to the following exact sequence 
$$
0 \longrightarrow \left(\dfrac{\Gamma}{\langle \scS\rangle}\right)_{\tor} \longrightarrow \dfrac{\Gamma}{\langle \scS \rangle} \longrightarrow 
\  \dfrac{\Gamma}{\langle \scS \rangle}/\left(\dfrac{\Gamma}{\langle \scS\rangle}\right)_{\tor} 
\longrightarrow 0,
$$ 
we obtain
 \begin{equation}
 \label{eq:gr-iso}
H_{\scS,  G}/\scC_\scS^\textbf{1} \simeq \Hom((\Gamma/\langle\scS\rangle)_{\tor}, G).
 \end{equation}
Furthermore, $E_k$ induces  the endomorphism $\widetilde{E_k}$ of $\Hom((\Gamma/\langle\scS\rangle)_{\tor}, G)$ with
 \begin{equation}
 \label{eq:tilde}
 \ker(\widetilde{E_k})=\Hom((\Gamma/\langle\scS\rangle)_{\tor},  G[k] ).
 \end{equation}
Here $ G[k] = \{x \in  G \mid x^k=1\} \simeq \Z/k\Z$. 
Combining \eqref{eq:bar}, \eqref{eq:gr-iso} and \eqref{eq:tilde} we get $\# \left( \mathrm{cc}(H_{\scS,  G})\cap L[k] \right) = m(\scS; \Z/k\Z)$, as desired. 
\end{proof}

\begin{definition}\quad
\label{def:k-partial}
\begin{enumerate}[(1)]
\item  
For each $k\in \Z$, the \emph{$k$-partial} intersection poset of $\A(G)$ is defined by
 \begin{equation*}
 \label{eq:par-subposet}
L^{\part}[k]:= \{\scC \in L^{\part} \mid \textbf{1} \in E_k(\scC)\}.
 \end{equation*}
\item 
The \emph{$k$-partial} characteristic polynomial of $\A( G)$ is defined by 
$$\chi_{\A( G)}^{\textrm{\textit{k}-par}}(t):= \sum_{\scC \in L^{\part}[k]} \mu(T^\scC, \scC)t^{\mathrm{dim}(\scC)}.$$ 
\end{enumerate}
\end{definition}

  \begin{theorem}
\label{thm:2nd-interpret-quasi}
If $q \in \Z_{>0}$, then 
$$\chi_{\A( G)}^{\textnormal{\textit{q}-par}}(q)=\chi^{\quasi}_{\A}(q).$$ 
\end{theorem}
 \begin{proof}
Very similar to Proof of Theorem \ref{thm:Lie} including the use of Lemma \ref{lem:k-components}.
\end{proof}
 \begin{corollary} 
\label{thm:2nd-interpret-constituent}
If $1 \le k \le\rho_{\A}$, then 
$$\chi_{\A( G)}^{\textnormal{\textit{k}-par}}(t)=f^k_{\A}(t).$$ 
  \end{corollary}

 \begin{corollary} 
\label{thm:2nd-interpret-total}
If $q \in \Z_{>0}$ and $1 \le k \le\rho_{\A}$, then 
\begin{align*}
 \chi_{\A( G)}^{\textnormal{\textit{q}-tot}}(q) = & \chi^{\quasi}_{\A\smallsetminus \A^{\tor}}(q), \\
\chi_{\A( G)}^{\textnormal{\textit{k}-tot}}(t) = & f^k_{\A\smallsetminus \A^{\tor}}(t). 
\end{align*}
  \end{corollary}

By Theorem \ref{thm:2nd-interpret-quasi}, we can write 
 \begin{equation}
 \label{eq:quasi-minus}
\chi^{\quasi}_{\A}(q)=\sum_{j=r_\Gamma-r_{\A}}^{r_\Gamma}(-1)^{r_\Gamma-j}\beta_j(q)q^j,
 \end{equation}
with each coefficient $\beta_j(q)$ is a periodic function given by
$$\beta_j(q) = (-1)^{r_\Gamma-j} \sum_{\substack{\scC \in L^{\part}(q) \\  j=\mathrm{dim}(\scC)}}  \mu(T^\scC, \scC) \ge 0.$$ 
It is easily seen that if $a, b\in\Z_{>0}$ and $a \mid b$, then $L^{\part}(a) \subseteq L^{\part}(b)$. 
This obvious inclusion between the subposets implies the result in \cite[Theorem 1.2]{CW12} about the inequality of the constituent coefficients. 
 \begin{corollary}[\cite{CW12}] 
\label{thm:compare-coe}
 If $a, b$ are positive integers and $a$ divides $b$, then for all $j$ with $r_\Gamma-r_{\A} \le j \le r_\Gamma$,
$$0 \le \beta_j(a)  \le \beta_j(b).$$
  \end{corollary}
  
   \begin{proof}
Note that $\mu$ strictly alternates in sign (Corollary \ref{cor:alternate-sign}).
\end{proof}

\begin{example}
\label{ex:caculation-toric}
Let $\Gamma=\Z^2$, $\A=\{\alpha, \beta, \gamma\}\subsetneq\Z^2$ with $\alpha=(-1, {1})$, $\beta=(0, {2})$, and $\gamma=(0, 4)$ as in Example \ref{ex:caculation-special}.
The Hasse diagrams of $L^{\part}[k]$ are drawn in Figures \ref{Fig7a}, \ref{Fig7b}, \ref{Fig7c}. 
The $k$-partial characteristic polynomials $\chi_{\A( G)}^{\textnormal{\textit{k}-par}}(t)$ are computed according to the subposets extracted from $L$. 
 \begin{figure}[h] 
\begin{subfigure}{.8\textwidth}
 \[
 \xymatrixrowsep{1cm}
\xymatrixcolsep{.1cm}
 \xymatrix{
\{(1,1)\}& & & \\
\{z_2=1\} \ar@{-}[u]  \ar@{-}[drr] &  &  &   &  \{z_1^{-1}z_2=1\}    \ar@{-}[ullll]  \ar@{-}[dll] \\
& & G^2& &
}
\]\end{subfigure}%
\caption{$\chi_{\A( G)}^{\textnormal{1-par}}(t)=t^2-2t+1=f^1_{\A}(t).$}
\label{Fig7a}
\end{figure}
 \begin{figure}[h]     
\begin{subfigure}{.6\textwidth}
 \[
 \xymatrixrowsep{1cm}
\xymatrixcolsep{.1cm}
 \xymatrix{
\{(1,1)\}& \{(-1,-1)\} & & \\
\{z_2=1\} \ar@{-}[u]  \ar@{-}[drr] & \{z_2=1\} \ar@{-}[u]  \ar@{-}[dr] &  &   &  \{z_1^{-1}z_2=1\}    \ar@{-}[ullll] \ar@{-}[ulll]   \ar@{-}[dll] \\
& & G^2& &
}
\]\end{subfigure}%
\caption{$\chi_{\A( G)}^{\textnormal{2-par}}(t)=t^2-3t+2=f^2_{\A}(t).$}
\label{Fig7b}
\end{figure}
 \begin{figure}[h] 
\begin{subfigure}{.8\textwidth}
 \[
 \xymatrixrowsep{1cm}
\xymatrixcolsep{.1cm}
 \xymatrix{
\{(1,1)\}& \{(-1,-1)\} & \{(i,i)\}& \{(-i,-i)\}\\
\{z_2=1\} \ar@{-}[u]  \ar@{-}[drr] & \{z_2=1\} \ar@{-}[u]  \ar@{-}[dr] & \{z_2=i\} \ar@{-}[u]  \ar@{-}[d] & \{z_2=-i\} \ar@{-}[u]  \ar@{-}[dl]  &  \{z_1^{-1}z_2=1\}    \ar@{-}[ullll] \ \ar@{-}[ulll] \ \ar@{-}[ull] \ar@{-}[ul]  \ar@{-}[dll] \\
& & G^2& &
}
\]\end{subfigure}%
\caption{$\chi_{\A( G)}^{\textnormal{4-par}}(t)=t^2-5t+4=f^4_{\A}(t).$}
\label{Fig7c}
\end{figure}
 \end{example}

\noindent
\textbf{Acknowledgements:} 
 TNT gratefully acknowledges the support of the scholarship program of 
the Japanese Ministry of Education, Culture, Sports, Science, and Technology 
(MEXT) under grant number 142506. 
MY is partially supported by JSPS KAKENHI Grant Numbers JP16K13741, JP15KK0144, JP18H01115. 
\bibliographystyle{alpha} 
\bibliography{references}

\end{document}